\documentclass[11pt]{amsart}

\usepackage{euscript}
\usepackage{amsmath}
\usepackage{amsfonts}
\usepackage{amssymb}
\usepackage{amsthm}
\usepackage{epsfig}
\usepackage{epstopdf}
\usepackage{url}
\usepackage{array}
\usepackage{amssymb}\usepackage{amscd}
\usepackage{epic}
\usepackage{eufrak}
\usepackage{tikz,underscore}
\usepackage{tikz-cd}
\usepackage{float}
\usepackage{subfigure}
\usepackage{tkz-euclide}

\usepackage{graphicx}
\usepackage{epstopdf}
\numberwithin{equation}{section}
\theoremstyle{plain}
\newtheorem{theorem}{Theorem}[section]

\newtheorem{proposition}[theorem]{Proposition}

\newtheorem{corollary}[theorem]{Corollary}

\newtheorem{ques}[theorem]{Question}

\theoremstyle{definition}

\newtheorem{remark}[theorem]{Remark}

\DeclareMathOperator{\Isom}{{\mathrm Isom}}

\def\geo{\partial_{\infty}}

\def\H{\mathbb{H}}

\title{Kleinian groups via strict hyperbolization}

\author{Beibei Liu}
\address{B.L.: School of Mathematics, Georgia Tech, Atlanta, GA, USA, 30332}
\email{bliu96@gatech.edu}

\begin{document}
\begin{abstract}
In this paper, we construct Kleinian groups $\Gamma<\Isom(\H^{2n})$ from the direct product of $n$ copies of the rank 2 free group $F_2$ via strict hyperbolization. We give a description of the limit set and its topological dimension. Such construction can be generalized to other right-angled Artin groups. 
\end{abstract}

\maketitle

\section{Introduction}
A Kleinian group is a discrete isometry group of the $n$-dimensional hyperbolic space $\H^{n}$. There are a lot of ways to construct Kleinian groups. The most common ones are to use the Poincar\'e fundamental polyhedron theorem (see e.g. \cite{Mas, Rat}), the Klein-Maskit Combination Theorem (see e.g. \cite{KAG, Mas}), and to construct arithmetic groups and their subgroups (see e.g. \cite{MR, VS}). One can also deform a given Kleinian group or to find limits of sequences of Kleinian groups (see, e.g. \cite{Kap} for a survey). In this paper, we construct Kleinian groups from the direct product of $n$ copies of the rank 2 free group $F_2\times \cdots \times F_2=F^{n}_{2}$ via the strict hyperbolization. 

The strict hyperbolization introduced by Charney and Davis is a  procedure which associates to a simplicial complex $K$ a piecewise hyperbolic space $\mathcal{G}_{X}(K)$ of curvature $\leq -1$ \cite{RM}. One can use one of Gromov's techniques to construct a polyhedron $\mathcal{H}(K)$ associated to $K$ which is a cubical cell complex where each cube is isometric to a regular Euclidean cube \cite{Gro}. The key ingredient in the strict hyperbolization procedure is to replace the Euclidean cube in $\mathcal{H}(K)$ by an appropriate face of some compact, connected, orientable, hyperbolic manifold $X$ with corner. The manifold $X$ is obtained by cutting an arithmetic hyperbolic manifold $M$ along a system of codimension one submanifolds \cite{RM}. Our construction of the Klenian groups in this paper relies on the arithmetic hyperbolic manifold $M$.  In particular, we need to take some finite cover of the manifold $M$ if necessary to ensure the normal injectivity radii of some closed geodesics are large enough. For simplicity, we still denote the finite cover by $M$. 

 The  direct product of $n$ copies of the rank $2$ free group $F_{2}$ is the fundamental group of the direct product of $n$ copies of a wedge of two circles, which we denote by $W^{n}$. The complex $W^{n}$ is actually the \emph{Salvetti complex} defined for the right-angled Artin group $F^{n}_{2}$. We refer the reader to the note \cite{Cha} for an introduction of right-angled Artin groups and the Salvetti complex. The $n$-dimensional complex $W^{n}$ corresponds to an $n$-dimensional  complex $Z$ embedded in the $2n$-dimensional arithmetic manifold $M^{2n}$ used in  the strict hyperbolization. The inclusion map $f': Z\rightarrow M^{2n}$ induces a map $f'_{\ast}: \pi_{1}(Z)\rightarrow \pi_{1}(M^{2n})$. We prove: 
 
 \begin{theorem}
\label{injective1}
The map $f'_{\ast}: \pi_{1}(Z)\rightarrow \pi_{1}(M^{2n})$ is injective. Hence $\Gamma_{n}=\pi_{1}(Z)$ is a torsion free Kleinian group, i.e. a discrete torison free isometry subgroup of $\Isom(\H^{2n})$. 
\end{theorem}

In general, one can use the method in the paper to construct Kleinian groups from other right-angled Artin groups such as the free abelian group $Z^{m}$ or the right-angled Artin groups represented by the $m$-gons. 

In contrast to Kleinian groups in dimension 3, there is no comprehensive structure theory for higher dimensional Kleinian groups, i.e. $n\geq 4$. One way to study higher dimensional Klenian groups is to see the geometric and topological properties of the limit set which is the accumulation set of an orbit in the visual boundary. For example, groups with zero dimensional limit sets are relatively well understood. We refer the readers to \cite{Kap} for more details about the study of higher dimensional Kleinian groups.  The limit set of the Kleinian groups $\Gamma_{n}$ we construct in the paper  is the closure of countably infinite many $(n-1)$-dimensional spheres $S^{n-1}$,  i.e. 
$$\Lambda(\Gamma_{n})=\mathcal{S}\cup E$$
where $\mathcal{S}$ is the union of the $(n-1)$-dimensional spheres $S^{n-1}$, and $E$ is the rest of points in the limit set with the cardinality of the continuum. The points in the limit set are endpoints of piecewise geodesic rays which are uniform quasi-geodesics. For the detailed description, see Section \ref{sec:construction} and Section \ref{sec:limitset}.

The Kleinian group $\Gamma_{n}$ is constructed via the right-angled Artin group $F^{n}_{2}$.  The boundary of $F^{n}_{2}$ here is defined to be the visual boundary of the universal cover of $W^{n}$, which is the join of $n$ copies of the Cantor set. In fact, the boundary of $F^{n}_{2}$ is well-defined  independent of choice of the CAT(0) space on which the group acts geometrically \cite{Rua}. This does not hold for general right-angled Artin groups \cite{CK}. On the contrary, for any nonelementary Kleinian group $\Gamma$, its limit set $\Lambda(\Gamma)$ cannot be a join of Cantor sets. However, we prove that the limit set $\Lambda(\Gamma_{n})$ contains the join of $n$ copies of $K_{3}$ where $K_{3}$ is a set of three points, hence it cannot  be embedded in $\mathbb{R}^{2n-2}$,  see  \cite[Lemma 9]{BKK}.

\begin{theorem}
\label{obsemb}
The limit set $\Lambda(\Gamma_{n})$ cannot be embedded in $\mathbb{R}^{2n-2}$. 
\end{theorem}

On the other hand, It is interesting to ask what properties of the set $\geo F^{n}_{2}$ is preserved in $\Lambda(\Gamma_{n})$.  For example one can ask: 
\begin{ques}
Whether the support of the simplicial homology of $\geo F^{n}_{2}$ is the same as the support  of simplicial homology (or \v{C}ech cohomology)  of $\Lambda(\Gamma_{n})$. 
\end{ques}

The homology group $H_{i}(\geo F^{n}_{2})$  is nonzero if and only if $i=0$ or $n-1$. We prove that:

%We can not prove whether the homology of  limit set $\Lambda(\Gamma_{n})$ has the same support, but show that the homology of the subset $\mathcal{S}\subset \Lambda(\Gamma_{n})$ does not vanish if and only if $i=0$ or $n-1$.  It is interesting to ask what properties in limit set $\Lambda(A)$ is preserved in $\Lambda(\Gamma)$. For example one can ask 
%\begin{ques}
%Whether the support of the simplicial homology of $\Lambda(A)$ is the same as the support  of simplicial homology (or Cech cohomology) of the limit set of $\Lambda(\Gamma_{n})$. 
%\end{ques}

%We prove that: 

\begin{theorem}
\label{topdim}
The topological dimension of $\Lambda(\Gamma)$ equals  $n-1$. 
\end{theorem}

\begin{corollary}
For any $n\geq 2$,  $H_{n-1}(\Lambda(\Gamma_{n}))$ is nontrivial. 
\end{corollary}

For $0<i<n$, we cannot determine whether $H_{i}(\Lambda(\Gamma_{n}))$ vanishes or not.

{\bf Acknowledgements.} I deeply appreciate Grigori Avramidi for introducing this interesting question to me and the helpful discussions during the project. I  want  to thank Tam Nguyen-Phan for her useful suggestions and discussions of the proof of Theorem \ref{injective1}.  I also would like to thank Dan Margalit for his comments on the earlier draft. The project began during my visit in Max Planck Institute for Mathematics in Bonn, and I am grateful to the institute  for its hospitality and financial support.

%\section{Backgounds}

%\subsection{Quasi-geodesics}

%\begin{definition}
%A map $q: I \to X$ defined on an interval $I\subset \mathbb{R}$ is called a $(A,\alpha)$-quasigeodesic (for $A\ge 1$ and $\alpha\ge 0$)  if 
%$$
% A^{-1} |s-t| - \alpha \le d(q(s), q(t))\le A |s-t| + \alpha
%$$
%for all $s, t\in I$. 
%\end{definition}

%\begin{proposition}\cite[Proposition 7.3]{KL1}
%\label{qua}
%Define the function 
%$$L(\theta, \varepsilon)= 2\cosh^{-1}\left( \dfrac{ e^{2}+1}{2\sin (\alpha/2)}\right)+1$$
%where $\alpha=\min\{ \theta, \pi/2-\arcsin(1/ \cosh \varepsilon)\}.$

%Suppose that $\gamma=\gamma_{1}\ast \cdots \ast \gamma_{n}\subseteq \bar{X}$ is a piecewise geodesic path from $x$ to $y$ such that:
%\begin{enumerate}
%\item Each geodesic arc $\gamma_{j}$ has length either at least $\varepsilon>0$ or at least $L=L(\theta, \varepsilon)$. 

%\item 
%If $\gamma_{j}$ has length $<L$, then the adjacent geodesic arcs $\gamma_{j-1}$ and $\gamma_{j+1}$ have lengths at least $L$ and  $\gamma_{j}$ meets $\gamma_{j-1}$ and $\gamma_{j+1}$ at angles 
 %$\geq \pi /2$.

%\item
%Other adjacent geodesic arcs meet at an angle $\geq \theta$. 

%\end{enumerate}
%Then $\gamma$ is a (2L, 4L+3)-quasigeodesic. 
%\end{proposition}

%\begin{theorem}\cite[Morse Lemma]{}
%Let $\gamma$ be an $(A, \alpha)$-quasigeodesic from $x$ to $y$. Then the Hasusdorff distance between $\gamma$ and $xy$ is uniformly bounded by a constant $C(A, \alpha)$ which only depends on $A, \alpha$. 
%\end{theorem}

\section{Strict hyperbolization}
In this section, we review the strict hyperbolization introduced by Charney and Davis \cite{RM}, which is used to construct the higher dimensional Kleinian groups $\Gamma_{n}<\Isom(\H^{2n})$ corresponding to the right-angled Artin group $F^{n}_{2}$. 

Let $B_{n}$ denote the symmetric group of the $n$-dimensional cube, and let $r_{i}$ denote the linear reflection across the hyperplane $x_{i}=0$ in $\mathbb{R}^{n}$. The group $B_{n}$ has standard action on $\mathbb{R}^{n}$ generated by permutations of coordinates and the reflections $r_i$. 

\begin{theorem}\cite[Theorem 6.1]{RM}
\label{arima}
For each $n\geq 0$, there is a closed connected hyperbolic $n$-dimensional manifold $M^{n}$, a system $\mathcal{Y}=\{Y_1, \cdots, Y_n\}$ of closed connected submanifolds of codimensional one in $M^{n}$, and an isometric action of $B_{n}$ on $M^{n}$, stabilizing $\mathcal{Y}$, such that the following properties hold: 
\begin{enumerate}
\item $Y_i$ is a component of the fixed point set of $r_i$ on $M^{n}$. 

\item Each $Y_i$ is totally geodesic in $M^{n}$. 

\item The $Y'_{i}$s intersect orthogonally. 

\item $Y_1\cap \cdots \cap Y_n$ is a single point $y$. 

\item $B_n$ fixes $y$ and the representation of $B_n$ on $T_{y}M^{n}$ is equivalent to the standard representation. 

\item $M^{n}$, as well as each $Y_i$ is orientable.

\end{enumerate}
\end{theorem}

\begin{remark}
The group $B_{n}$ normalizes the group $\pi_{1}(M^{n})$ by the construction of $M^{n}$, see \cite[Section 6]{RM}. Hence, it acts on $M^{n}$ as isometries. The key point to ensure $Y_1\cap \cdots \cap Y_n$ is a single point is to prove that $\pi_{1}(M^{n})$ is a torsion free congruence subgroup of some cocompact lattice $O(\phi)<O(n, 1)$. For details,  see \cite[Lemma 6.6]{RM}. 
\end{remark}

%Now we recall the arithemtic construction of $M^{n}$ briefly. Let $K=\mathbb{Q}(\sqrt{d})$ be a totally real quadratic extension of the rationals  and let $A$ denote the ring of algebraic integers in $K$. Choose $\epsilon$ in $A$ so that $\epsilon>0$ and $\bar{\epsilon}<0$. Defini a symmetric bilinear form on $A^{n+1}$ by 
%\[  \phi(e_i, e_j)=
%\begin{cases} 
%      \delta_{ij} & (i, j)\neq (0, 0) \\
%      -\epsilon & (i, j)=(0, 0).
%   \end{cases}
%\]
%Denote the isometry group of $\phi$ by $O(\phi)$. It is a discrete and cocompact subgroup of $O(n, 1)$. Then $O(\phi)/ \{\pm 1\}$ is a discrete cocompact lattice of $\Isom(\H^{n})$. Let $\Gamma<SO(n, 1)\cap O(\phi)$ be a torsion free subgroup. It is easy to see $M'=\H^{n}/ \Gamma$ satisfies all the conditions in Theorem \ref{arima} except (4). In order to construct closed hyperbolic manifolds satisfying condition (4), one can use the congruence subgroups \cite{RM}. Let $\mathfrak{p}$ be a prime ideal in $A$. Then $A/ \mathfrak{p}$ is a finite field and the form $\phi$ induces a nonsingular bilinear form $\phi_{\mathfrak{p}}$ on $(A/\mathfrak{p})^{n+1}$. Its isometry group $O(\phi_{\mathfrak{p}})$ is finite. The kernel of the natural projection $P: O(\phi)\rightarrow O(\phi_{\mathfrak{p}})$ is denoted by $\Gamma(\mathfrak{p})$ and are called \emph{congruence subgroup } of $O(\phi)$. 

%%\label{interone}
%%\end{lemma}

\begin{proposition}
\label{normalinj}
Given a constant $R>0$, there exists a closed connected hyperbolic $n$-dimensional manifold $M'$ and a closed geodesic $\gamma\subset M'$ such that the normal injective radius of $\gamma$ in $M'$ is at least $R$, and  $M'$ satisfies  all the conditions in Theorem \ref{arima}.

\end{proposition}

\begin{proof}
We start from the discrete cocompact lattice $\Gamma=SO(n, 1)\cap O(\phi)$ which is used to construct the arithmetic manifold $M^{n}$ in Theorem \ref{arima}. Choose a loxodromic isometry  $g\in \Gamma$,  and  let $A\subset \H^{n}$ be an axis  of $g$ such that $\gamma_{1}=A/ \langle g \rangle$ is a closed geodesic in $M^{n}=\H^{n}/ \Gamma$.  Let $C$ denote the $R$-neighborhood of $\gamma_{1}$ in $M^{n}$. There are only finitely many isometris $h_{i}\in \Gamma$ such that $h_{i}(C)\cap C\neq \emptyset$ where $i\in \{1, \cdots, m\}$ and $h_{i}\notin \langle g \rangle$. One can choose a congruence subgroup $\Gamma'< \Gamma$ which does not contain any of the finitely many commutators $[g, h_{i}]$. Then none of $h_{i}$ is a power of $g$ in the quotient $\Gamma/ \Gamma'$. The subgroup $\Gamma'$ does not contain any of the finitely many isometries $h_{i}$, but contains some power of $g$ (say $g^{k}$). Since the subgroup $\Gamma'$ is a congruence subgroup, the quotient manifold $M'=\H^{n}/ \Gamma'$ satisfies all the conditions in Theorem \ref{arima}. The closed geodesic $\gamma=A/ \langle g^{k} \rangle$ has normal injective radius at least $R$ in $M'$.

%In order to check that $M'$ satisfies all the properties in Theorem \ref{arima}, it suffices to check property (4). Note that $\Gamma'$ is a finite index subgroup of the congruence subgroup $\Gamma$. It is also a torsion-free congruence subgroup. Hence $M'$ satisfies all thse properties by Lemma \ref{interone}. 

%\begin{figure}[H]
%\centering
%\begin{tikzcd}
%X=\H^{n}/ \langle g \rangle \arrow[r, ] \arrow[dr, ]
%& \H^{n}/\Gamma'=M' \arrow[d]\\
%& \H^{n}/ \Gamma=M
%\end{tikzcd}
%\end{figure}

\end{proof}

By the same argument, we have:

\begin{corollary}
\label{coro:inj}
Given a constant $R>0$ and $m>0$, there exists a closed connected hyperbolic $n$-dimensional manifold $M'$ and closed geodesics $\gamma_1, \cdots, \gamma_{m}\subset M'$ such that their normal injective radii in $M'$ are all at least $R$, and $M'$ satisfies all the conditions in Theorem \ref{arima}. 
\end{corollary}

%Heve we follow the notation in \cite{RM}. Suppose that $M^{n}$ is a smooth manifold and that $Y^{n-1}$ is a smooth submanifold of codimension one. To \emph{cut open along $Y$} and replaces it by the normal $S^{0}$-normal bundle of $Y$ in $M$. The result is denoted by $M\odot Y$. If $\mathcal{Y}=\{Y_1, \cdots, Y_k \}$ is a system of codimension one submanifolds of $M$, one can iterate the above cutting-open construction to obtain a smooth manifold with corners 
%$$M\odot \mathcal{Y}=M\odot Y_1 \odot Y_2\odot \cdots  \odot Y_k$$
%together with a projection map $\pi: M \odot \mathcal{Y}\rightarrow M$. 

%Let $X^{n}=M'\odot \mathcal{Y}$. 

%\begin{corollary}\cite{RM}
%\label{cut}
%For each $n>0$, there is a compact, connected, orientable hyperbolic $n$-manifold with corner $X^{n}$ together with an action of $B_n$ on $X^n$ by isometries so that the following properties hold.
%\begin{enumerate}
%\item The poset of faces of $X^n$ is $B_n$-equivariantly isomorphic to the poset of faces of $\square^{n}$.

%\item Each face of $X^{n}$ is totally geodesic. 

%\item The faces of $X^{n}$ intersect orthogonally. 

%\item Each $0$-diemnsional face is a single point.

%\item The map $f: X^{n}\rightarrow \square^{n}$ is degree one as its restriction to each face of $X^{n}$. 
%\end{enumerate}

%\end{corollary}

\section{The construction of Kleinian groups}
\label{sec:construction}

Let $W$ be the wedge of two circles whose fundamental group is the rank 2 free group $F_{2}$ generated by $v_{0}, v_{1}$. Let $a_{0}, a_{1}$ denote the two circles in $W$. There is a map 
$$f: W^{n}=\underbrace{W\times W \times \cdots \times W}_{n}\rightarrow T^{2n}$$
 where $T^{2n}$ is the $2n$-dimensional torus. Note that $W^{n}$ is a $n$-dimensional CAT(0) complex and the map $f$ induces  $f_{\ast}: \pi_{1}(W^{n})=F^{n}_{2}\rightarrow \pi_{1}(T^{2n})$.  Note that $f_{\ast}$ is  not injective since $F^{n}_{2}$ is not commutative.

 %By corollary \ref{cut}, the $2n$-dimensional closed hyperbolic manifold $M$ can be cut along $\mathcal{Y}$ such that the poset of faces of $X^{2n}$ is $B_{2n}$-equivariantly isomorphic to the poset of faces of $\square^{2n}$. 
 
We claim that  $W^{n}$ corresponds to an $n$-dimensional CAT(0) complex $Z$ embedded in the  arithmetic manifold $M$ in Theorem \ref{arima}. 

Recall that for the arithmetic manifold $M$, there is a smooth map $\phi: M\rightarrow T^{2n}$  such that $\mathcal{Y}$ is the transverse inverse image of the standard system of subtori in $T^{2n}$ \cite[Lemma 5.3]{RM}. Let $\Sigma_{i_{1}\cdots i_{n}}=\phi^{-1}(f(a_{i_{1}}\times a_{i_{2}}\times \cdots \times a_{i_{n}}))$ where $i_{j}\in \{0, 1\}$.  Note that $f(a_{i_{1}}\times \cdots \times a_{i_{n}})=T^{n}\subset T^{2n}$, and the submanifold $\Sigma_{i_1\cdots i_n}$ is an $n$-dimensional totally geodesic submanifold in $M$. There are  $2^{n}$ such $n$-dimensional submanifolds, and the intersection of any two such submanifolds is a totally geodesic submanifold $\phi^{-1}(T^{m})$ for some $0\leq m \leq n-1$. The intersection of all these $2^{n}$ submanifolds is a single point, denoted by $O\in M$. We let $Z$ denote the union of these $2^{n}$ totally geodesic submanifolds $\Sigma_{i_1\cdots i_n}$.

%Let $\gamma_{j}, \gamma'_{j}$ denote the inverse image of $f(a_{i_{j}})$ under $\phi$. They are closed geodesics in $M$. 

By the discussion above,  if  the subscripts of two submanifolds $\Sigma_{i_1\cdots i_n}$ and $\Sigma_{i'_{1}\cdots i'_{n}}$ are different except at the $k$th-entry, i.e. $i_j\neq i'_j$ for $j\neq k$, and $i_k=i'_k$. Then  the intersection of $\Sigma_{i_1\cdots i_n}$ and $\Sigma_{i'_{1}\cdots i'_{n}}$ is a closed geodesic, denote by $\gamma_k$ or $\gamma'_k$ depending on the $k$-entry is $0$ or $1$. Then we have $2n$ such closed geodesics and they intersect orthogonally with each other.  By Corollary \ref{coro:inj}, we assume that the normal injectivity radii of the $2n$ closed geodesics $\gamma_i, \gamma'_i$   are all  at least $3L$ where $L=2\cosh^{-1}(2\sqrt{2})+1$ in $M$  up to some finite index. The constant $L$ is the same constant as the one in \cite[Proposition 7.2]{KL} by letting $\theta=\pi/2$.

%Let $a_i, b_i$ denote  the two circles in the $i$-th copy of $W$ in $W^{n}$. For each element like $a_1\times b_2\times \cdots \times a_n\subset W\times \cdots \times W$, it corresponds to an $n$-dimensional face $T^{n}$ in $T^{2n}$, which corresponds to a totally geodesic $n$-dimensional submanifold $N$ in $M$. Then the complex $Z$ is the union of $2^{n}$ such $n$-dimensional submanifolds. Note that the intersection of two submanifolds is a totally geodesic submanifold of dimension no greater than $n-1$, and the intersection of all these submanifolds is a single point, denoted by $O$.
 
 %Let $\Sigma_{01\cdots 0}$ denote an $n$-dimensional submanifold where the $0$ (or $1$) in the $i$-th entry of the subscript means that it corresponds to the circle $a_i$ (or $b_i$).

In analogue to the map $f$, there exists a map $f': Z \rightarrow M$ and an induced map $f'_{\ast}: \pi_{1}(Z)\rightarrow \pi_{1}(M)=\Gamma$. In contrast to the map $f_{\ast}$, we prove that $f'_{\ast}$ is injective. 

%\begin{proposition}
%\label{localinter}
%If the normal injective radii of $\gamma_{i}, \gamma'_{i}$ are large, then there is only one local model for the intersection point. 
%\end{proposition}

%\begin{proof}
%Fix a lift $\tilde{O}$ of $O$ in $\H^{2n}$, which we can assume it is the center of the unit disk model. Let $\tilde{\Sigma}_{i_1\cdots i_n}$ be the lift of the submanifold $\Sigma_{i_1\cdots i_n}$ passing through $\tilde{O}$. By the construction, 
%$$\bigcap_{i_1\cdots i_n\in \{0, 1\}^{n}} \tilde{\Sigma}_{i_1\cdots i_n}=\tilde{O}.$$
%We claim that for $\gamma\in \Gamma_{n}$, and two different lifts $=\tilde{\Sigma}=\tilde{\Sigma}_{i_1\cdots i_n}, \tilde{\Sigma'}=\tilde{\Sigma}_{i'_1\cdots i'_n}$,
%$$\gamma(\tilde{\Sigma})\cap \gamma (\tilde{\Sigma'})=\gamma(\tilde{\Sigma}\cap \tilde{\Sigma'}), \quad \tilde{\Sigma}\cap \gamma(\tilde{\Sigma'})=\emptyset.$$ 
%We first study the case that $\gamma(\tilde{\Sigma})\neq \tilde{\Sigma}$. Suppose that 

%\end{proof}

\medskip
\noindent
{\bf Proof of Theorem \ref{injective1}: }
Pick an element $\omega\in \pi_{1}(Z, O)$ corresponding to the geodesic loop $w$. Then we can write $\omega=\omega_{1}\omega_2\cdots \omega_{k}$ where $\omega_{i}$ is in $\pi_{1}(\Sigma_{i}, O)$ and $\Sigma_i$ is one of the $2^{n}$ submanifolds $\Sigma_{i_1\cdots i_n}$. Write the loop $w=w_1\ast w_2\ast \cdots \ast w_k$ where each $w_i$ is a geodesic loop in $\Sigma_{i}$ based on $O$. Now we prove that $f'_{\ast}(\omega)$ is not the identity for any nontrivial element $\omega$. 

Consider the universal cover $\H^{2n}$ of $M$, and a lift $\tilde{w}$  of the geodesic loop $w$ in $\H^{2n}$. The bi-infinite path $\tilde{w}=\tilde{w}_{1}\ast \tilde{w}_{2}\ast \cdots \tilde{w}_{k}\ast\cdots$ is a  piecewise geodesic path such that each segment $\tilde{w}_{i}$ is a geodesic segment in a lift $\tilde{\Sigma}_{i}$  of $\Sigma_{i}$. Note that two consecutive segments $\tilde{w}_{i}$ and $\tilde{w}_{i+1}$ meet at one lift $\tilde{O}_{i}$ of $O$, and the intersection of  the corresponding  lifts of $\Sigma_i$ and $\Sigma_{i+1}$ is either a single point $\tilde{O}_{i}$ or contains a geodesic which is a lift of one of the closed geodesics $\gamma_k, \gamma'_k$.  Moreover, these two lifts intersect orthogonally \cite[Corollary 6.2]{RM}.

 If the two lifts $\tilde{\Sigma}_{i}$ and $\tilde{\Sigma}_{i+1}$ intersect at a single point $\tilde{O}_{i}$, then the consecutive geodesic segments $\tilde{w}_{i}$ and $\tilde{w}_{i+1}$ meet at $\tilde{O}_{i}$ with angle $\pi/2$. If instead the intersection of these lifts contains a geodesic $\tilde{\gamma_{i}}$, the angle between $\tilde{w}_{i}$ and $\tilde{w}_{i+1}$ can be arbitrarily small as in Figure \ref{piecewisegeo}.  In this case, we need to replace $\tilde{w}_{i+1}$ by a new path which is homotopic to $\tilde{w}_{i+1}$. There are two cases depending on the intersection of the lifts $\tilde{\Sigma}_{i+1}$ and $\tilde{\Sigma}_{i+2}$ (a lift where $\tilde{w}_{i+2}$ lies in):

Case (1):  Suppose that $\tilde{\Sigma}_{i+1}$ intersects $\tilde{\Sigma}_{i+2}$ at a single point $\tilde{O}_{i+1}$. Then we replace $\tilde{w}_{i+1}$ by the path $\tilde{O}_{i}a_{i+1}\ast a_{i+1}\tilde{O}_{i+1}$ where $\tilde{O}_{i+1}a_{i+1}$ is perpendicular to $\tilde{\gamma}_{i}$ at $a_{i+1}$. 

Case (2): The intersection of $\tilde{\Sigma}_{i+1}$ and $\tilde{\Sigma}_{i+2}$ contains a geodesic $\tilde{\gamma}_{i+1}$. Note that both  $\tilde{O}_{i+1}$ and $\tilde{O}_{i}$ are lifts of $O$, so there exists an element $g\in \Gamma$ such that $\tilde{O}_{i+1}=g(\tilde{O}_{i})$. Consider the geodesic $g(\tilde{\gamma}_{i})$. Then $\tilde{\gamma}_{i+1}$ either is identified with $g(\tilde{\gamma}_{i})$ or intersects $g(\tilde{\gamma}_{i+1})$ orthogonally at $\tilde{O}_{i+1}$.  Let $a_{i}b_{i}$ denote the shortest geodesic which is othorgonal to both $\tilde{\gamma}_{i}$ and $g(\tilde{\gamma}_{i})$. Then we replace $\tilde{w}_{i+1}$ by $\tilde{O}_{i}a_{i}\ast a_{i}b_{i}\ast b_{i}\tilde{O}_{i+1}$,  which is homotopic to $\tilde{w}_{i+1}$. Note that the distance between $\tilde{\gamma}_{i}$ and $g(\tilde{\gamma}_{i})$ is at least $3L$ by construction. Hence the length of $a_{i}b_{i}$ is at least $3L$. By repeating this process for each segment $\tilde{w}_{i}$,  we replace the bi-infinite path $\tilde{w}$ by a new piecewise geodesic path $\tilde{w}'$ which is homotopic to $\tilde{w}$. We claim that $\tilde{w}'$ is a quasi-geodesic. By the Morse lemma, the isometry $\omega$ represented by $\tilde{w}'$ and $\tilde{w}$ is nontrivial. 

%Hence, the projection of $\ti$ homotopolically nontrivial and therefore, $\tilde{w}$ is also homotopically nontivial.  

Observe that the piecewise geodesic path $\tilde{w}'$ contains long geodesics which arise from the large normal injectivity radii of the geodesics $\gamma_k, \gamma'_k$.  The remaining geodesic segments  might be very short. Note that every two consecutive geodesic segments meet at the angle $\pi/2$ by the construction. These short geodesic segments locally look like the ones (e.g. $b_i\tilde{O}_{i+1}, \tilde{O}_{i+1}a_{i+1}$) in Figure \ref{piecewisegeo}. In the figure, the green segments denote the long segments while the red ones denote the short segments. Note that the length of the red segment could be zero depending on the intersection of the lifts $\tilde{\Sigma}_{i+1}$ and $\tilde{\Sigma}_{i+2}$. 

\begin{figure}[H]
\centering
\includegraphics[width=3.0in]{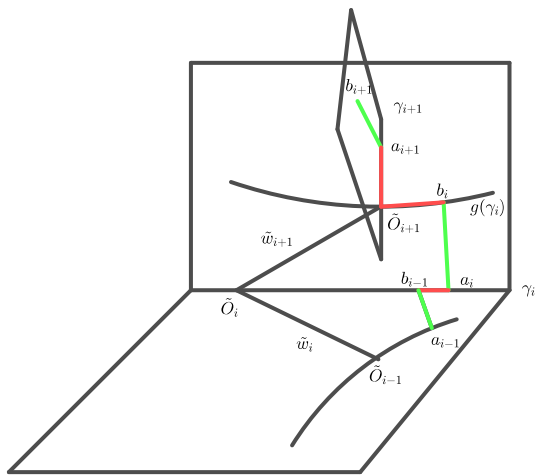}
\caption{peicewise geodesic path \label{piecewisegeo}}
\end{figure}

Suppose that the length of the long geodesic segments $3L\leq d(a_{i}, b_{i})< 6L$. Actually if $d(a_{i}, b_{i})\geq 6L$, take the point $a_{i1}\in a_{i}b_{i}$ such that $d(a_{i}, a_{i1})=3L$. If $d(a_{i1}, b_{i})\geq 6L$, we continue the process until we get $a_{ij}\in a_{i}b_{i}$ such that $3L\leq d(a_{ij}, b_{i})<6L$. Thus we get a new partition of the piecewise geodesic path $\tilde{w}'$ such that the long geodesic segment has length in $[3L, 6L)$, and consecutive arcs meet either at the angle $\pi$ or the angle $\pi/2$. For the short arcs, if the length is $\geq L$, the path $\tilde{w}'$ is a uniform quasi-geodesic by \cite[Proposition 7.2]{KL}. Hence we assume that the lengths of some short segments are $<L$. 

In order to prove that $\tilde{w}'$ is $(A, B)$-quasigeodesic, with $A\geq 1$ and $B\geq 0$, we need to verify the inequality that 
$$\dfrac{1}{A} length(\tilde{w}'\mid_{[t_{a}, t_{b}]})-B\leq d(a, b)\leq A\cdot length(\tilde{w}'\mid_{[t_{a}, t_{b}]})+B$$
for all pair of points $a, b\in \tilde{w}'$ where $\tilde{w}'(t_{a})=a$ and $\tilde{w}'(t_{b})=b$. The upper bound (for arbitrary $A\geq 1$ and $B\geq 0$) follows from the triangle inequality and we only need to establish the lower bound. 

 Consider the subpath $a_{i}b_{i}\ast b_{i}\tilde{O}_{i+1}\ast \tilde{O}_{i+1}a_{i+1}\ast a_{i+1}b_{i+1}$ with long segments $a_{i}b_{i}$ and $a_{i+1}b_{i+1}$. By the triangle inequality, 
 $$d(a_{i}, a_{i+1})\geq 3L-L-L=L.$$
 Observe that $a_{i+1}b_{i+1}$ and $a_{i}a_{i+1}$ meet at $a_{i+1}$ with angle $\pi/2$. The bisectors of the arc $a_{i+1}b_{i+1}$ and $a_{i}a_{i+1}$ are at least distance $2$ apart by the similar argument of \cite[Proposition 7.2]{KL}. 
 
Suppose that the points $a=\tilde{w}'(t_{a}), b=\tilde{w'}(t_{b})$ in $\tilde{w'}$ are terminal points of geodesic segments $\tilde{w'}_{i}, \tilde{w'}_{j}, i<j$. Note that $\tilde{w}'|_{[t_a, t_b]}$ contains at least $(j-i-4)/3$ long geodesic segments like $a_{i}b_{i}$ in Figure \ref{piecewisegeo}. If two consecutive segments are both long segments, by the proof of  \cite[Proposition 7.2]{KL}, their bisectors are at least  distance $2$ apart. If there are $2$ short geodesic segments (e.g. $b_{i}\tilde{O}_{i+1}, \tilde{O}_{i+1}a_{i+1}$) lying between the long geodesic segments (e.g. $a_{i}b_{i}$ and $a_{i+1}b_{i+1}$), then we consider the bisectors of $a_{i}a_{i+1}$ and $a_{i+1}b_{i+1}$ which are also at least distance 2 apart. Every pair of these bisectors  divides $ab$ into a small segment with length  at least $2$. By adding these lengths together, we obtain the inequality 
 $$d(a, b)\geq \dfrac{2}{6}(j-i-4),$$
 while 
 $$length(\tilde{w}'|_{[t_{a}, t_{b}]})\leq 6(j-i+1)L.$$
 putting these inequalities together, we obtain
 $$d(a, b)\geq \dfrac{1}{18L} length(\tilde{w}'|_{[t_{a}, t_{b}]})-4.$$
 Lastly, for general points $a, b\in \tilde{w}'_{i}, \tilde{w}'_{j}$, they are within distance $<6L$ from the terminal endpoints of $a', b'$ of these segments.  Hence, 
 $$d(a, b)\geq d(a', b')-12L\geq \dfrac{1}{18L} length (\tilde{w}'|_{[t_{a'}, t_{b'}]})-4-12L\geq \dfrac{1}{18L} length(\tilde{w}'|_{[t_{a}, t_{b}]})-(4+12L).$$

\qed

\section{The limit set of the Kleinian group}
\label{sec:limitset}

By Theorem \ref{injective1},  $\Gamma_{n}=\pi_{1}(Z)$ is a Kleinian group. In this  section, we study the properties of the limit set $\Lambda(\Gamma_{n})$. Recall that $Z$ is a $CAT(0)$-complex, and we let $\tilde{Z}$ denote its universal cover which is the union of the  lifts of the $n$-dimensional  submanifolds $\Sigma_{i_1\cdots i_n}$ in $\H^{2n}$.

Given a point $\tilde{O}\in \tilde{Z}$, two geodesic rays $\rho_{1}$ and $\rho_{2}$ in $\tilde{Z}$ are called \emph{asymptotic} if they are at finite Hausdorff distance. The \emph{ideal boundary} of the metric space $\tilde{Z}$ is the collection of equivalence classes of geodesic rays, and we denote it by $\partial Z_{CAT(0)}$.

% Let $\partial Z_{CAT(0)}$ denote the visual boundary of the universal cover $\tilde{Z}$.  Given a point $x\in \tilde{Z}$  a point $\xi\in \partial Z_{CAT(0)}$, there exists a geodesic ray $\rho$ with initial point $x$ such that $\rho(\infty)=\xi$ when $Z$ is regarded as a CAT(0)-space. Two geodesic rays $\rho_{1}$ and $\rho_{2}$ in $\tilde{Z}$ are called \emph{asymptotic} if they are at finite Hausdorff distance. The \emph{ideal boundary} of a metric space $\tilde{Z}$ is the collection of equivalence classes of geodesic rays. 

The visual topology $\tau'_{\tilde{O}}$ on $\partial Z_{CAT(0)}$ is generated by the basis of neighborhoods
$$\{ N'(\rho, \epsilon, R)\mid \rho \in \partial_{\tilde{O}}Z, \epsilon>0, R>0\},$$
where 
$$N'(\rho, \epsilon , R)=\{\rho': d(\rho(R), \rho'(R))<\epsilon \}, \textup{ with } R\gg 1, \epsilon \ll 1$$ 
and
$$ \partial_{\tilde{O}}Z: =\{\rho: \rho \textup{ is a geodesic ray in } \tilde{Z} \textup{ with } \rho(0)=\tilde{O}\}. $$

Fix a lift $\tilde{O}$ of $O$ in $\tilde{Z}$. Consider the lifts of the $n$-dimensional submanifolds $\Sigma_{i_1\cdots i_{n}}$ passing through $\tilde{O}$. Each lift is one copy of the $n$-dimensional plane $\H^{n}$ whose visual boundary is $S^{n-1}$. We let $S_{i_1\cdots i_n}$ denote the visual boundary of the lift of $\Sigma_{i_1\cdots i_n}$ passing through $\tilde{O}$. Let $\mathbb{S}$ be the union of the $2^{n}$ spheres $S_{i_1\cdots i_n}$. Then we have 
$$\partial Z_{CAT(0)}=\overline{\Gamma(\mathbb{S})}=\Gamma(\mathbb{S})\cup E$$
where $E$ denote the remaining points not in $\Gamma(\mathbb{S})$.

Each point in the visual boundary corresponds to a geodesic ray $\rho$ emanating from $\tilde{O}$. By the construction of the universal cover, the geodesic rays travel along the lifts of the submanifolds $\Sigma_{i_1\cdots i_n}$. The points in $\Gamma(\mathbb{S})$  correspond to the geodesic rays that stay in one lift after some time $t$. Otherwise, if the geodesic rays keep travelling along different lifts as $t\rightarrow \infty$, the endpoints lie in $E$.

%Then the limit set $\Lambda(\Gamma)$ is the closure of the orbit $\Gamma(\mathbb{S})$. We write 
%$$\Lambda(\Gamma)=\Gamma(\mathbb{S})\cup E$$
%where $E$ denote the extra points in the closure, which has the cardinality of continuum. It is easy to see that $\Gamma(\mathbb{S})$ contributes to $H_{n-1}(\Lambda(\Gamma))$, but we are not sure whether it is trivial or not.

%Let $\partial Z_{CAT(0)}$ denote the visual boundary of $Z$ as a CAT(0) complex, and $\partial Z_{\H^{2n}}$ denote the visual boundary of  $Z$ embedded in $\H^{2n}$, which is also the limit set $\Lambda(\Gamma)$. Then 

There is a natural surjective map $i: \partial Z_{CAT(0)}\rightarrow \partial Z_{\H^{2n}}$ where $\partial Z_{\H^{2n}}$ denotes the visual boundary of $\tilde{Z}$ embedded in $\H^{2n}$, and actually this is the same as the limit set $\Lambda(\Gamma_{n})$ by Theorem \ref{injective1}. We first prove that $\partial  Z_{CAT(0)}$ cannot be embedded in $\partial \H^{m}$ for any $m<2n$, see Theorem \ref{obsemb}. We also compare $\partial Z_{CAT(0)}$ and $\partial Z_{\H^{2n}}$, proving that $i$ is a homeomorphism, see Theorem \ref{injective2}.

\medskip
\noindent
{\bf Proof of Theorem \ref{obsemb}: }Recall that a finite graph is planar if and only if it does not contain a subgraph that is a subdivision of the complete graph $K_{5}$ or the complete biparite graph $K_{3, 3}$, which is known as Kuratowski's theorem. In general, the complex $\ast^{n}K_{3}$ which is the join of $n$ copies of three points $K_3$  cannot be embedded in $\mathbb{R}^{2n-2}$, see  \cite[Lemma 9]{BKK}. It suffices to prove that the limit set $\Lambda(\Gamma_{n})$ contains the complex $\ast^{n} K_{3}$.

 We first consider the case that $n=2$. Recall that every lift of the surface $\Sigma_{i_{1}i_{2}}$ is a copy of the $2$-dimensional plane $\H^{2}$ with ideal boundary $S^{1}$ where $i_{j}\in \{0, 1\}$. Let $\tilde{O}$ denote one lift of $O$. The configuration of the ideal boundary of the lifts of the four surfaces $\Sigma_{i_{1}i_2}$ passing through $\tilde{O}$ is shown as in Figure \ref{circles}, and the limit set $\Lambda(\Gamma_{2})$ contains this configuration. It is not hard  to see that  in Figure \ref{circles}, the vertices $A, B, C, D, E, F$ consist of a complete biparite graph $K_{3, 3}$, i.e. $K_{3}\ast K_{3}$, hence, it is not planar. Therefore, $\Lambda(\Gamma_{2})$ cannot be embedded in $R^{2}$. 
 
 We next use the induction on $n$ to show that the lifts of the $2^{n}$ $n$-dimensional submanifolds $\Sigma_{i_1\cdots i_n}$ passing through $\tilde{O}$ contains the subcomplex $*^{n}K_3$. Assume the claim holds for $n-1$. Recall the lifts of an $n$-dimensional submanifold $\Sigma_{i_{1}\cdots i_{n}}$ are copies of the $n$-dimensional planes $\H^{n}$ with ideal boundary $S^{n-1}$ where $i_{j}\in \{ 0, 1\}$. Consider the ideal boundary $S_{i_1\cdots i_{n-1}0}$ of the lifts of the submanifolds $\Sigma_{i_1\cdots i_{n-1}0}$ passing through $\tilde{O}$. By the construction of the complex $Z$ in Section \ref{sec:construction}, the intersection $\bigcap S_{i_1\cdots i_{n-1}0}$ consists of two points $A, A'$ which are the endpoints of  the lift of closed geodesic $\gamma_{n}$ passing through $\tilde{O}$. By the assumption of the induction, $\bigcup S_{i_{1}\cdots i_{n-1}0}$ contains $\ast^{n-1} K_{3}$, which indicates that it also contains the set $\ast^{n-1}K_{3}* \{ A, A' \}$. By the same reason, $\bigcap S_{i_{1}\cdots s_{n-1} 1}$ consists of two points $B, B'$ which are the endpoints of the lift of the closed geodesic $\gamma'_{n}$ passing through $\tilde{O}$,  and $\bigcup S_{i_{1}\cdots i_{n-1}1}$ contains the complex $\ast^{n-1} K_{3}$. Hence, $\bigcup S_{i_{1}\cdots i_{n}}$ contains $\ast^{n-1}K_{3} * \{ A, A' , B, B' \}$, therefore it contains the complex $\ast^{n} K_{3}$.

\qed

\begin{figure}[H]
\centering
\includegraphics[width=3.0in]{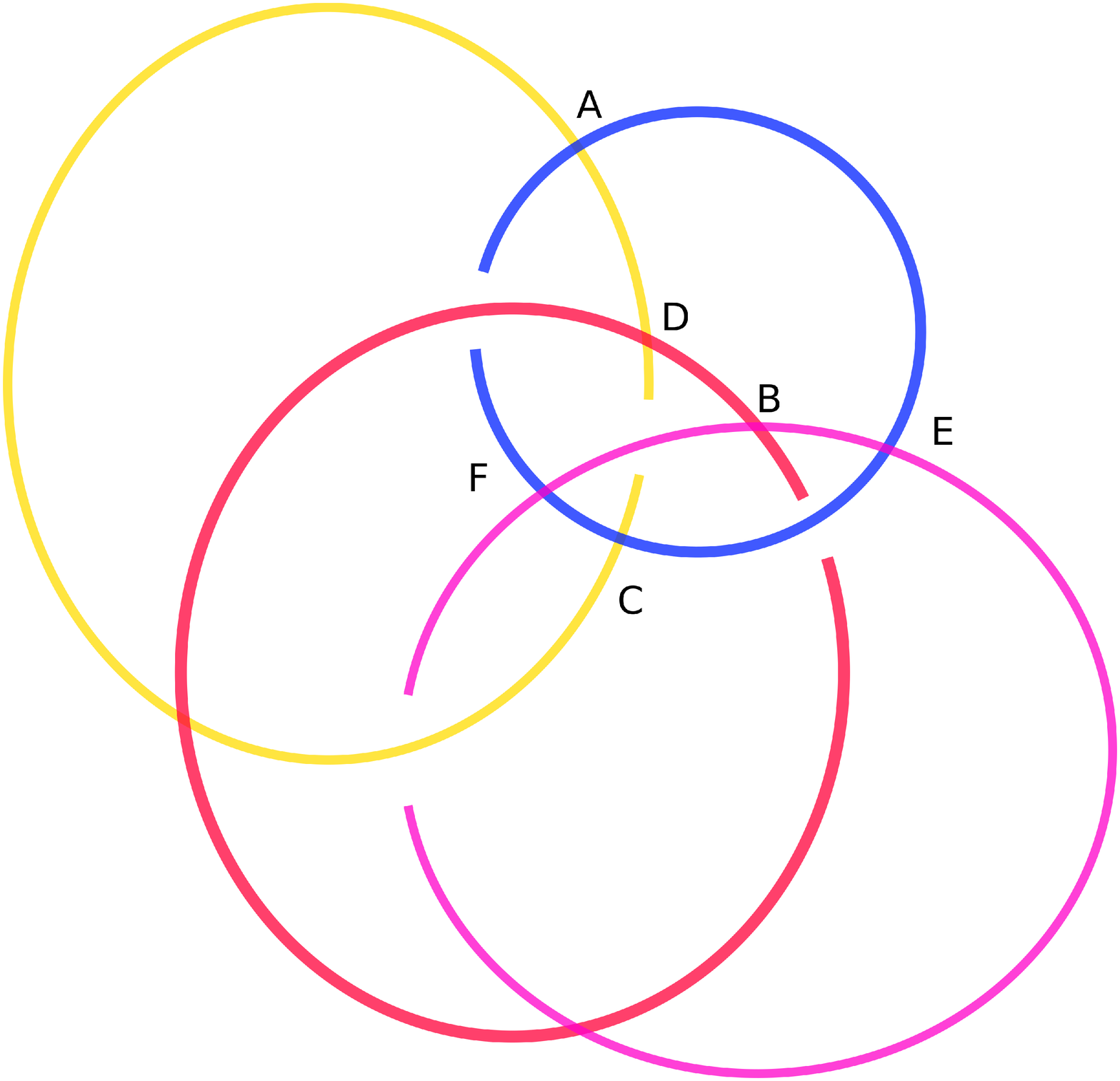}
\caption{local link  \label{circles}}

\end{figure}

\begin{theorem}
\label{injective2}
The map $i: \partial Z_{CAT(0)}\rightarrow \partial Z_{\H^{2n}}$ is homeomorphic.  
\end{theorem}

\begin{proof}

%When $Z$ is regarded as a CAT(0)-space, the geodesic ray $\rho$ has two types. The first one keeps branching between the n-planes which are universal covers of the $2^{n}$ submanifolds at first and stays at the same plane for the rest of the time. The second types keeps branching until $r\rightarrow\infty$.  Note that for both types of rays $\rho$, its pieces consist of geodesic segments in the $n$-planes. So we can replace the segments as we do in Theorem \ref{injective1} to get a new path $\rho'$ which is a uniform quasi-geodesic. It is easy to observe that $\rho(\infty)=\rho'(\infty)$ by the $\delta$-hyperbolicity of each $n$-planes.  

We first prove that the  map $i$ is injective. Consider two  different geodesic rays $\rho_{1}$ and $\rho_{2}$ with different endpoints $\xi_{1}, \xi_{2}\in \partial Z_{CAT(0)}$.   We first suppose that  both $\rho_{1}(t), \rho_{2}(t)$ keep staying  in some lifts of the submanifolds $\Sigma_{i_1\cdots i_n}$, i.e. $\rho_{1}|_{[t, \infty]}, \rho_{2}|_{[t, \infty]}$ stay in some copies of hyperbolic planes $\H^{n}$,  respectively. Suppose that $\xi_{1}=\xi_{2}\in \partial \H^{2n}$. Then the intersection of these two lifts is non-empty, and we let $A$ denote one intersection point. Then the geodesic ray $A\xi$ lies in both of the lifts. By the $\delta$-hyperbolicity, there exists a constant number $K>0$ such that the geodesic rays $\rho_{1}|_{[t, \infty]}$, $\rho_{2}|_{[t, \infty]}$ are within the $K$-neighborhoods of $A\xi$.  Hence, the Hausdorff distance of $\rho_{1}|_{[t, \infty]}$ and $\rho_{2}|_{[t, \infty]}$ in the CAT(0) complex $Z$ is bounded by $2K$ which contradicts to our assumption that $\xi_{1}\neq \xi_{2}\in \partial Z_{CAT(0)}$. 

%Hence, there exists a constant number $K>0$ such that the geodesic rays $\rho_{1}|_{[t, \infty]}$, $\rho_{2}|_{[t, \infty]}$ are within the $K$-neighborhoods of $\tilde{O}\xi_{1}$ and $\tilde{O}\xi_{2}$ in $\H^{2n}$, respectively.  If $\xi_{1}=\xi_{2}$ in $\partial \H^{2n}$, the Hausdorff distance of $\rho_{1}|_{[t, \infty]}$ and $\rho$

Now we consider the case that  $\rho_{1}$ keeps travelling along different lifts as $t\rightarrow \infty$. Suppose that there exists a lift $P_0$ such that the intersections $\rho_1\cap P_0$ and $\rho_{2}\cap P_0$ are nonempty and $\rho_1, \rho_2$ won't stay in the same lift after $P_0$. Let $t_{0}$ denote the time when the geodesic ray $\rho_1$ starts to enter another lift $P_1$ different from $P_{0}$. Note that $\rho_2$ may enter another lift $P_{2}\neq P_1$ which we call  type 2 or stay in the same lift $P_0$ for the rest of the time which we call type 1.  We claim that  in both cases,  we form a new bi-infinite piecewise geodesic path which is a quasi-geodesic with endpoints $\xi_1, \xi_2\in \partial \H^{2n}$. By the Morse lemma, $\xi_1\neq \xi_2$ in $\partial \H^{2n}$. 

The intersection of $P_1$ and $P_0$ is either a point or contains a geodesic $\tilde{\gamma}_i$ as in the proof of Theorem \ref{injective1}. Assume that $\rho_{2}$ is type 1. Then we make a new bi-infinite piecewise geodesic path 
$$\rho_{3}=\rho_{1}|_{[t_{0}, \infty)}\ast ab \ast b\rho_{2}(\infty)$$
where $a=\rho_{1}(t_{0})$ and $b$ is the unique intersection point of $P_0$ and $P_1$ or $b\rho_{2}(\infty)$ meets $\tilde{\gamma}_i$ orthogonally at $b$. By replacing the segments in $\rho_3$ as what we do in Theorem \ref{injective1}, we get the new path $\rho'_3$ which is a uniform quasi-geodesic. By $\delta$-hyperbolicity of the lifts, $\rho_3$ is within bounded neighborhood of $\rho'_3$. Hence, $\rho_{3}(\infty)=\rho'_{3}(\infty)$ and $\rho_{3}(-\infty)=\rho'_{3}(-\infty)$. By the Morse lemma, $\rho'_{3}(\infty)$ is different from $\rho'_{3}(-\infty)$ which means that $\rho_{1}(\infty)$ is different from $\rho_{2}(\infty)$. 

If $\rho_{2}$ is type 2, assume that $\rho_{2}(t'_{0})$ is the starting point of the geodesic segment in $P_{2}$. We form an bi-infinite piecewise geodesic path
$$\rho_{3}=\rho_{1}|_{[t_{0}, \infty)}\ast ab \ast \rho_{2}|_{[t'_{0}, \infty)}$$
where $a=\rho_{1}(t_{0})$ and $b=\rho_{2}(t'_{0})$. By the similar argument above, we have  a new bi-infinite piecewise geodesic path $\rho'_3$ which is a quasi-geodesic and $\rho'_3(\infty)=\rho_3(\infty), \rho'_3(-\infty)=\rho_3(\infty)$. Hence, the two endpoints of $\rho_1$ and $\rho_{2}$ are different. 

We last check the case that both the geodesic rays $\rho_{1}, \rho_{2}$ are type 2, and they travel along the same lifts for all $t\in [0, \infty)$. By the similar argument to the previous case, there are  piecewise geodesic paths $\rho'_{1}, \rho'_{2}$ which are both quasi-geodesics in $\H^{2n}$ such that $\rho_{i}$ is within a uniform bounded neighborhood of $\rho'_i$ for $i=1, 2$. Hence, $\rho'_{1}(\infty)=\rho_{1}(\infty)$ and $\rho'_{2}(\infty)=\rho_{2}(\infty)$.  If $\xi_{1}=\xi_{2}$ in $\partial X_{\H^{2n}}$, then $\tilde{O}\xi_{1}=\tilde{O}\xi_{2}$, and the Hausdorff distance between $\rho'_{1}$ and $\rho'_{2}$ is also uniformly bounded by the Morse lemma. Therefore, there exist sufficiently large  time $t_{1}, t_{2}, t'_{1}, t'_{2}$ such that $\rho_{1}|_{[t_{1}, t_{2}]}$ and $\rho_{2}|_{[t'_{1}, t'_{2}]}$ both lie in the same lift and the Hausdorff distance between these segments in the CAT(0)-complex is uniformly bounded which contradicts to the assumption that $\rho_{1}$ and $\rho_{2}$ are geodesic rays which are not equivalent in the  CAT(0)-space. 

We have proved that the map $i$ is one-to-one. It suffices to prove that the inverse map $i^{-}: \partial Z_{\H^{2n}}\rightarrow \partial Z_{CAT(0)}$ is continuous in order to see that $i$ is a homeomorphism since both $\partial Z_{\H^{2n}}$ and  $Z_{CAT(0)}$ are compact sets. We briefly recall the visual topology $\tau'_{\tilde{O}}$ on $\partial Z_{CAT(0)}$, which is generated by the basis of neighborhood 
$$\{ N'(\rho, \epsilon, R)\mid \rho \in \partial_{\tilde{O}}Z, \epsilon>0, R>0\}.$$
%%$$N(\rho, \epsilon , R)=\{\rho': d(\rho(R), \rho'(R))<\epsilon \}, \textup{ with } R\gg 1, \epsilon \ll 1$$ 
%and
%$$ \partial_{\tilde{O}}Z: =\{\rho: \rho \textup{ is a geodesic ray in } Z \textup{ with } \rho(0)=\tilde{O}\}. $$

%By the argument above, it is easy to see that $Z$ is a Gromov hyperbolic space when it is regarded as a CAT(0)-space with $\delta'$-hyperbolicity. Recall that given a number $k>2\delta'$, one can define the topology $\tau'_{k}$ on $\partial Z_{CAT(0)}$, where the basis of neighborhoods of a point $\rho(\infty)$ given by 
%$$U'_{k, m}:=\{\rho': \textup{ dist}(\rho'(t), \rho(t))<k, \quad t\in [0, m] \}, \quad m\in \mathbb{R}_{+}. $$
One similarly defines the topology $\tau_{\tilde{O}}$ on $\partial Z_{\H^{2n}}$. Without loss of generality, we assume that the rays $\rho(t)$ in the CAT(0) space $Z$ are uniform quasi-geodesics in $\H^{2n}$ and they are within  uniform neighborhoods of geodesic rays $\rho(0)\rho(\infty)$. This means that $i$ maps $N'(\rho, \epsilon, R)$ to an open set $N(\rho(0)\rho(\infty), \epsilon', R)$ in $\partial Z_{\H^{2n}}$, which indicates that $i^{-}$ is continuous.

\end{proof}

\begin{corollary}
\label{topdim}
The topological dimension of $\Lambda(\Gamma_n)$  equals $n-1$. 
\end{corollary}

\begin{proof}
Note that the topological dimension of $\partial Z_{CAT(0)}$ equals $n-1$ \cite[Theorem 1.7]{Best}. Then the corollary follows straightforward from Theorem \ref{injective2}. 
\end{proof}

\begin{corollary}
\label{coro:homology}
For any $n\geq 2$,  $H_{n-1}(\Lambda(\Gamma_{n}))$ is nontrivial. 
\end{corollary}

\begin{proof}
Recall that any $(n-1)$-dimensional sphere $S_{i_{1}\cdots i_{n}}$ generates a cycle, and it is nontrivial in homology  by Corollary \ref{topdim}. 

\end{proof}

\bibliographystyle{abbrv}
\bibliography{biblio}

\begin{thebibliography}{10}

\bibitem{Best}
M.~Bestvina.
\newblock Local homology properties of boundaries of groups.
\newblock {\em Michigan Math. J.}, 43(1):123--139, 1996.

\bibitem{BKK}
M.~Bestvina, M.~Kapovich, and B.~Kleiner.
\newblock Van {K}ampen's embedding obstruction for discrete groups.
\newblock {\em Invent. Math.}, 150(2):219--235, 2002.

\bibitem{Cha}
R.~Charney.
\newblock An introduction to right-angled {A}rtin groups.
\newblock {\em Geom. Dedicata}, 125:141--158, 2007.

\bibitem{RM}
R.~M. Charney and M.~W. Davis.
\newblock Strict hyperbolization.
\newblock {\em Topology}, 34(2):329--350, 1995.

\bibitem{CK}
C.~B. Croke and B.~Kleiner.
\newblock Spaces with nonpositive curvature and their ideal boundaries.
\newblock {\em Topology}, 39(3):549--556, 2000.

\bibitem{Gro}
M.~Gromov.
\newblock Hyperbolic groups.
\newblock In {\em Essays in group theory}, volume~8 of {\em Math. Sci. Res.
  Inst. Publ.}, pages 75--263. Springer, New York, 1987.

\bibitem{Kap}
M.~Kapovich.
\newblock Kleinian groups in higher dimensions.
\newblock In {\em Geometry and dynamics of groups and spaces}, volume 265 of
  {\em Progr. Math.}, pages 487--564. Birkh\"{a}user, Basel, 2008.

\bibitem{KL}
M.~Kapovich and B.~Liu.
\newblock Geometric finiteness in negatively pinched {H}adamard manifolds.
\newblock {\em Ann. Acad. Sci. Fenn. Math.}, 44(2):841--875, 2019.

\bibitem{KAG}
S.~L. Krushkal, B.~N. Apanasov, and N.~A. Gusevski\u{\i}.
\newblock {\em Kleinian groups and uniformization in examples and problems},
  volume~62 of {\em Translations of Mathematical Monographs}.
\newblock American Mathematical Society, Providence, RI, 1986.
\newblock Translated from the Russian by H. H. McFaden, Translation edited and
  with a preface by Bernard Maskit.

\bibitem{MR}
C.~Maclachlan and A.~W. Reid.
\newblock {\em The arithmetic of hyperbolic 3-manifolds}, volume 219 of {\em
  Graduate Texts in Mathematics}.
\newblock Springer-Verlag, New York, 2003.

\bibitem{Mas}
B.~Maskit.
\newblock {\em Kleinian groups}, volume 287 of {\em Grundlehren der
  Mathematischen Wissenschaften [Fundamental Principles of Mathematical
  Sciences]}.
\newblock Springer-Verlag, Berlin, 1988.

\bibitem{Rat}
J.~G. Ratcliffe.
\newblock {\em Foundations of hyperbolic manifolds}, volume 149 of {\em
  Graduate Texts in Mathematics}.
\newblock Springer, Cham, third edition, [2019] \copyright 2019.

\bibitem{Rua}
K.~E. Ruane.
\newblock Boundaries of {${\rm CAT}(0)$} groups of the form {$\Gamma=G\times
  H$}.
\newblock {\em Topology Appl.}, 92(2):131--151, 1999.

\bibitem{VS}
E.~B. Vinberg and O.~V. Shvartsman.
\newblock Discrete groups of motions of spaces of constant curvature.
\newblock In {\em Geometry, {II}}, volume~29 of {\em Encyclopaedia Math. Sci.},
  pages 139--248. Springer, Berlin, 1993.

\end{thebibliography}

\end{document}